\documentclass[12pt]{amsart}
\usepackage{amssymb, enumerate, palatino}
\usepackage{amsfonts}
\usepackage{latexsym}
\usepackage{amscd}
\usepackage[mathscr]{euscript}

\newcommand{\N}{{\mathbb{N}}}
\newcommand{\Z}{{\mathbb{Z}}}

\newcommand{\V}{{\mathrm{V}}}

\newcommand{\Cs}{{C$^*$-algebra}}
\newcommand{\cK}{{\mathcal{K}}}
\newcommand{\cM}{{\mathcal{M}}}
\newcommand{\nprecsim}{{\operatorname{\hskip3pt \precsim\hskip-9pt
      |\hskip6pt}}}

\newcommand{\uloopr}[1]{\ar@'{@+{[0,0]+(-4,5)}@+{[0,0]+(0,10)}@+{[0,0] +(4,5)}}^{#1}}
\newcommand{\uloopd}[1]{\ar@'{@+{[0,0]+(5,4)}@+{[0,0]+(10,0)}@+{[0,0]+ (5,-4)}}^{#1}}
\newcommand{\dloopr}[1]{\ar@'{@+{[0,0]+(-4,-5)}@+{[0,0]+(0,-10)}@+{[0, 0]+(4,-5)}}_{#1}}
\newcommand{\dloopd}[1]{\ar@'{@+{[0,0]+(-5,4)}@+{[0,0]+(-10,0)}@+{[0,0 ]+(-5,-4)}}_{#1}}

\newcommand{\luloop}[1]{\ar@'{@+{[0,0]+(-8,2)}@+{[0,0]+(-10,10)}@+{[0, 0]+(2,2)}}^{#1}}

\newtheorem{lem}{Lemma}[section]
\newtheorem{corol}[lem]{Corollary}
\newtheorem{theor}[lem]{Theorem}
\newtheorem{prop}[lem]{Proposition}
\newtheorem{defi}[lem]{Definition}
\newtheorem{defis}[lem]{Definitions}
\theoremstyle{definition}

\newtheorem{exem}[lem]{Example}

\newtheorem{nota}[lem]{Notation}
\newtheorem{rema}[lem]{Remark}

\begin{document}

\title[The Corona Factorization Property and Refinement Monoids]{The Corona Factorization Property and Refinement Monoids}%
\author{Eduard Ortega}
\address{Department of Mathematics and Computer Science, University of Southern Denmark, Campusvej 55, DK-5230, Odense M, Denmark} \email{ortega@imada.sdu.dk}
\author{Francesc Perera}
\address{Departament de Matem\`atiques, Universitat Aut\`onoma de Barcelona, 08193 Bellaterra, Barcelona, Spain} \email{perera@mat.uab.cat}
\author{Mikael R\o rdam}
\address{Department of Mathematical Sciences, University of Copenhagen, Universitetsparken 5, DK-2100, Copenhagen \O, Denmark} \email{rordam@math.ku.dk}

\thanks{} 
\date{\today}

\begin{abstract} The Corona Factorization Property of a \Cs,
  originally defined to study extensions of \Cs s, has turned out to say
  something
  important about intrinsic structural properties of the \Cs. We
  show in this paper that a $\sigma$-unital \Cs{} $A$ of real rank
  zero has the Corona Factorization Property  if and only if its
  monoid $\V(A)$ of Murray-von Neumann equivalence classes of
  projections in matrix algebras over $A$ has a certain (rather weak)
  comparability property that we call the Corona Factorization Property
  (for monoids). We show that a projection in such a \Cs{} is properly
  infinite if (and only if) a multiple of it is properly infinite.

The latter result is obtained from some more general result we establish about
conical refinement monoids. We show that the set of order
units (together with the zero-element) in a conical refinement monoid
is again a refinement monoid under the assumption that the monoid
satisfies weak divisibility; and
if $u$ is an element in a refinement monoid such that $nu$ is properly
infinite, then $u$ can be written as a sum $u=s+t$ such that $ns$ and
$nt$ are properly infinite.
\end{abstract}
\maketitle

\section*{Introduction}

\noindent
The Corona Factorization Property was introduced by Elliott and
Kucerovsky in \cite{EK}, and studied further by Kucerovsky and Ng in
\cite{KN},  as a tool to study extensions of \Cs s. The salient
feature of the Corona Factorization Property  is that
it ensures that (full) extensions are absorbing.

At the level of \Cs s,  the Corona Factorization
Property is easily defined: a \Cs{} $A$ has this property if and only if every
full projection in the multiplier algebra of the stabilization,
$A\otimes \cK$, is properly infinite.

The failure of having the Corona Factorization Property seems to be a
common feature of several examples of ``badly behaved'' \Cs s that
have been exhibited over the last 5--10 years. The vague phrase
``badly behaved'' should perhaps be replaced with ``infinite
dimensional'' in a suitable non-commutative sense. The example by the
third named author in \cite{rdoc} of a (simple, nuclear) \Cs{} $A$
which is not stable, but where a matrix algebra over $A$ is stable,
does not have the Corona Factorization Property. In fact, the existence
of a full, non-properly infinite projection in the multiplier algebra
over the stabilization of $A$ is the crucial property in that
example. Likewise for the example, also by the third named author, in
\cite{Ro} of a simple nuclear \Cs{} that contains both a finite and an
infinite projection.

It remains unknown if there is an example of a simple \Cs{} \emph{of
  real rank zero} which contains both finite and infinite
projections. If not, then every simple \Cs{} of real rank zero is
either stably finite or purely infinite. It was observed by Zhang that
this dichotomy holds for  simple \Cs s of real rank zero that satisfy
the Corona Factorization Property.

Many (classes of) \Cs s are known to satisfy the Corona Factorization
Property. These include purely infinite C$^*$-algebras (simple or
not), $\mathcal{Z}$-stable algebras (see \cite{hrw}), simple
\Cs s with real rank zero, stable rank one and weak
unperforation on their $K_0$-group (so, in particular, all
simple AF-algebras), \cite{KN}, algebras of the type
$C(X)\otimes\mathcal{K}$, where $X$ is a finite dimensional compact
Hausdorff metric space, \cite{ppv}, and also all unital \Cs s with
finite decomposition rank, \cite{OR}.

We show that a
$\sigma$-unital \Cs{} $A$ of real rank zero (simple or not) satisfies the
Corona Factorization Property if and only if its monoid $\V(A)$ of
Murray-von Neumann equivalence classes of projections satisfies a
certain (rather weak) comparability property that we call the
\emph{Corona Factorization Property} (for monoids). We also
characterize the $\sigma$-unital \Cs s $A$ of real rank zero where every
ideal has the Corona Factorization Property as those where $\V(A)$
satisfies a (more natural) comparability property, that we call the
\emph{strong Corona Factorization Property} (for monoids).

We further show that \Cs s with the Corona Factorization
Property (simple or not) satisfy an analog of the dichotomy for simple
\Cs s of real rank zero with the Corona Factorization Property mentioned
above. This is first established at the level of monoids. More precisely, if
$M$ is a conical refinement monoid with the strong Corona
Factorization Property, and if $u$ is an element in $M$ such that a
multiple of $u$ is properly infinite, then $u$ itself is properly
infinite. (We say that $u$ is properly infinite if $2u \le u$.) If $M$
has the Corona Factorization Property, then the statement above only
holds for \emph{order units}. This result can then be translated
into a statement about \Cs s of real rank zero with the Corona
Factorization Property, saying that any full projection $p$ in the
stabilization of such a \Cs, for which some multiple $p \oplus p
\oplus \cdots \oplus p$ is properly infinite, is itself properly
infinite.

In outline, the paper is as follows. After recalling basic
definitions, we proceed in Section~2 to define the concept of weak
divisibility of order units, a property that implies that order
units can be split into a sum of several order units.
We prove that any conical
refinement monoid with a properly infinite order unit has weak
divisibility of order units. This includes $\V(A)$, for any properly
infinite \Cs{} $A$ with real rank zero.

Section~3 is devoted to showing that in a refinement
monoid $M$ with weak divisibility of order units, the submonoid
$M^*\cup\{0\}$ of order units together with zero is a simple
refinement monoid. This allows for the reduction of certain arguments to
the simple case. For example, we show that for a real rank zero
algebra $A$ such that $\V(A)$ has weak divisibility of order units and
strict unperforation, then the subsemigroup of classes of full
projections is cancellative. A further consequence is obtained in
Section~4, where we show that, in a conical refinement monoid with an
order unit $u$ such that $nu$ is properly infinite, $u=s+t$ for
order units $s$ and $t$ such that $ns $ and $nt$ are also properly
infinite.

Finally, in Section~5 we obtain the results, described above, on when
$\sigma$-unital \Cs s of real rank zero have the Corona Factorization
Property; and we show that (full) projections in such \Cs s are properly
infinite if some multiple of them are properly infinite.

\section{Preliminaries}

All monoids in this paper will be abelian, written additively. The
main object of study in this paper is the monoid of (Murray-von
Neumann) equivalence classes of projections associated to a
C$^*$-algebra. We briefly recall its construction below.

We say that two projections $p$ and $q$ in a C$^*$-algebra $A$ are Murray-von Neumann \emph{equivalent} provided there is a partial isometry $v$ such that $p=vv^*$ and $v^*v=q$. One can extend this to the set $\mathcal{P}(M_{\infty}(A))$ of projections in $M_{\infty}(A)$, which becomes a congruence, and thereby construct
\[
\V(A)=\mathcal{P}(M_{\infty}(A))/\!\!\sim\,,
\]
where $[p]\in \V(A)$ stands for the equivalence class that contains the projection $p$ in $M_{\infty}(A)$.
This becomes an abelian monoid with operation $[p]+[q]=[p\oplus q]$, where $p\oplus q$ refers to the matrix $\left(\begin{smallmatrix} p & 0 \\ 0 & q\end{smallmatrix}\right)$.

We say that a monoid $M$ is \textit{conical} if $x+y=0$ only when $x=y=0$. Notice that, for any C$^*$-algebra $A$, the projection monoid $\V(A)$ is conical. Therefore, we shall assume that all our monoids are conical although, for emphasis, this assumption will be repeated in our statements of results.

A class of C$^*$-algebras for which the monoid $\V(A)$ captures a good deal of their structure is that of real rank zero. Recall that $A$ has real rank zero provided that the set of self-adjoint, invertible elements is dense in the set of all self-adjoint elements (see \cite{bp}). It is well-known that this condition allows to produce projections on demand, as is equivalent, for example, to the statement asserting that each hereditary subalgebra contains an approximate unit consisting of projections.

As shown by Ara and Pardo (\cite{AP}), based on work of Zhang, for C$^*$-algebras with real rank zero, the projection monoid enjoys the additional property of having refinement (see, e.g. \cite{dob}, \cite{WE}), which we now proceed to define:

We say that a monoid $M$ is a \emph{refinement monoid} if whenever  $a_1+a_2=b_1+b_2$ in $M$, there exist elements $z_{ij}$ in $M$ such that $a_i=z_{1i}+z_{2i}$ and
$b_i=z_{i1}+z_{i2}$ for $i=1,2$. The elements $z_{ij}$ are referred to as a \emph{refinement} of the original equation. For ease of notation, we shall write these equations as follows
\[
\begin{array}{c|cc}
  & a_1& a_2  \\
\hline b_1 & z_{11}&z_{12} \\
  b_2 & z_{21}& z_{22}  \end{array}
\]
and refer to the above display as a refinement matrix.

All abelian monoids have a natural pre-order, the \emph{algebraic
  ordering}, defined as follows: if $x$, $y\in M$, we write $x\leq y$
if there is $z$ in $M$ such that $x+z=y$. In the case of $\V(A)$, the
algebraic ordering is given by Murray-von Neumann subequivalence, that
is, $[p]\leq [q]$ if and only if there is a projection $p'\leq q$ such
that $p\sim p'$. We also write, as is customary, $p\precsim q$ to mean
that $p$ is subequivalent to $q$.

\section{Weak divisibility}

The purpose of this section is to consider a weak form of divisibility
that appears quite frequently both in ring theory and operator
algebras (see \cite{peror} and \cite{AGPS}). In the sequel it will be
important to apply this to the set of full projections, which in the
monoid-theoretical context corresponds to the set of order units.

Let us recall that an element $u$ in a monoid $M$ is an \emph{order
  unit} provided $u\neq 0$ and, for any $x$ in $M$, there is
$n\in\mathbb{N}$ such that $x\leq nu$. In the case that every non-zero
element of $M$ is an order unit, then $M$ is said to be
\emph{simple}. Recall that a projection $p$ of a C$^*$-algebra $A$ is
\emph{full} if the closed, two-sided ideal it generates is $A$. This
is equivalent to saying that $[p]$ is an order unit as an element of
$\V(A)$. If $A$ is a simple C$^*$-algebra, then $\V(A)$ is a simple
monoid.

For a monoid $M$, let us denote by $M^*$ the subsemigroup of order
units of $M$ (so in the simple case, $M^*=M\setminus\{0\}$). Notice
that $M^*\cup\{0\}$ is a submonoid of $M$.

\begin{lem}\label{lemma_3_1}
Let $M$ be a conical refinement monoid. Given two order units $u$ and $v$ in $M$, there exists an order unit $w\in M$ such that $w\leq u,v$.
\end{lem}
\begin{proof} Since $u$ is an order unit, there exist $n\in \mathbb{N}$ and $t\in M$ such that $v+t=
nu$. It follows then from \cite[Lemma 1.9]{WE} that there is a refinement
$$\begin{array}{c|cccc}
 & u & u &\cdots & u \\
 \hline v & v_1 & v_2 & \cdots & v_n \\
 t & t_1 & t_2 & \cdots & t_n
\end{array}\,,$$
with $v_1\leq v_2\leq\cdots\leq v_n$ and $t_1\geq t_2\geq\cdots\geq
t_n$. Set $w=v_n$, which is an order unit since $v$
is an order unit and $v\leq n w$.
\end{proof}

\begin{defi}
Let $M$ be a monoid. An element $x$ in $M$ will be termed \emph{weakly
divisible} if there exist $a$ and $b$ in $M$ such that $x=2a+3b$.  We say
that $M$ is weakly divisible if every element is weakly divisible.
We say that $M$ has \emph{weak divisibility for order units} if every order unit is weakly divisible.
\end{defi}

Weak divisibility of order units immediately enables us to decompose order units into sums of order units, as the following lemma testifies:

\begin{lem}\label{lemma_3_2}
Let $M$ be a monoid that has weak divisibility of
order units. Given any order unit $u\in M$ there exist order units $v,
w\in M$ such that $u=v+w$ and $v\leq w$.
\end{lem}
\begin{proof}
By assumption, there are elements $x$ and $y$ in $M$ with $u=2x+3y$. Put $v=x+y$ and
$w=x+2y$. Then $v$ and $w$ are order units since $u\leq
3v$ and $u\leq 2w$. Clearly $u=v+w$ and $v\leq w$.
\end{proof}
\begin{nota}
For elements $x$, $y$ in a monoid $M$, we will use $x\le^* y$ to mean $x+z=y$ for an order unit $z$ in $M$. If $M$ is simple, then all we are asking for is that $z\neq 0$. In particular, in the simple case $x<y$ (taken to mean $x\leq y$ and $x\neq y$) implies $x\leq^*y$.
\end{nota}
Our observation above yields:
\begin{corol}
\label{cor:oubelow}
Let $M$ be a conical refinement monoid with weak divisibility of order units. Given two order units $u,v$ in $M$, there is an order unit $w$ such that $w\leq^*u,v$.
\end{corol}
\begin{proof}
According to Lemma \ref{lemma_3_2}, we may write $u=u_1+u_2$ and $v=v_1+v_2$, where $u_i$ and $v_i$ are order units. Apply now Lemma \ref{lemma_3_1} to find an order unit $w$ such that $w\leq u_2,v_2$, whence $w\leq^* u,v$.
\end{proof}

Recall that an ideal of a monoid $M$ is a submonoid $I$ which is
hereditary with respect to the algebraic ordering. In other words, for
$x$ and $y$ in $M$, we have $x+y\in I$ if and only if $x, y\in I$. For
a C$^*$-algebra $A$ and a closed, two-sided ideal $I$ of $A$, the
monoid $\V(I)$ naturally becomes an ideal of $\V(A)$. Once we have an
ideal, we can define a congruence on $M$ by declaring $x\sim y$ if
there are elements $a$, $b$ in $I$ such that $x+a=y+b$. Then
$M/I=M/\!\!\sim$ naturally becomes a monoid, whose elements will be
denoted by $\overline{x}$, for $x$ in $M$. In this language, simple
monoids are those which do not have non-trivial ideals.

If $M$ is a refinement monoid and $I$ is an ideal of $M$, then it is
easy to see that both $I$ and $M/I$ are refinement monoids (see
\cite{agop}). Again, for C$^*$-algebras of real rank zero, the natural
quotient map induces an isomorphism $\V(A)/\V(I)\cong \V(A/I)$ (see
\cite{agop}).

Recall that a non-zero element $x$ in a monoid $M$ is an \emph{atom}
(or \emph{irreducible}) if whenever $x=a+b$, then either $a=0$ or
$b=0$. A monoid without atoms is called \textit{atomless}.

Given an element $u$ in a monoid $M$, let us denote by $\langle u\rangle$
the submonoid generated by $u$, that is, $\langle
u\rangle=\{0,u,2u,\ldots\}$. Observe that
if $M$ is a refinement monoid, then $u$ is an atom in $M$ if and only
if $\langle u\rangle$
is an ideal of $M$ and $\langle
u\rangle$ is isomorphic to $\mathbb{Z}^+$ (via $u\mapsto 1$). It was
shown in \cite[Theorem 6.7]{AGPS} (see also
Remark~\ref{rem:PR} below) that an element in a conical refinement
monoid is weakly divisible if and only if it is not an atom in any
quotient. Combining this result with the observation made above, we
get the following:

\begin{prop}
\label{prop:wd}
Let $M$ be a conical refinement monoid. Then $M$ is weakly
divisible if and only if no non-zero quotient of $M$ has an
ideal isomorphic to $\mathbb{Z}^+$.
\end{prop}

As an immediate consequence, we obtain the following corollary:

\begin{corol}
\label{lem:simple}
A conical simple refinement monoid $M$ is atomless, and hence weakly
divisible, if and only if it
is not isomorphic to $\Z^+$.
\end{corol}

\begin{rema} \label{rem:PR}
In the countable case, which will be of interest as $\V(A)$ is
countable whenever $A$ is a separable C$^*$-algebra, Proposition
\ref{prop:wd} can be obtained by the arguments in \cite{peror}. We
briefly indicate how to proceed in that case. First, recall that a
dimension monoid is, by definition, an inductive limit of simplicial
monoids or, equivalently, a monoid that can be represented as $\V(A)$
for an AF-algebra $A$. Now, if $M$ is a countable refinement monoid,
use \cite[Theorem 3.9]{peror} to find a (countable) dimension monoid
$\Delta$ and a surjective map $\alpha\colon\Delta\to M$ such that
$x\propto y$ if and only if $\alpha(x)\propto\alpha (y)$. (Here,
$x\propto y$ means that $x\leq ny$ for some natural number $n$.) In
particular, such a map induces an isomorphism between the ideal
lattices of $\Delta$ and $M$. By the observation above, that $\langle u
\rangle$ is an ideal in $M$ and $\langle u
\rangle \cong \Z^+$ if and only if $u$ is an atom in $M$, one
can see that the
map $\alpha$ also induces a one-to-one correspondence between
atoms in $\Delta$ and atoms in $M$.  This allows to reduce the problem
to dimension monoids. And in that case the result holds after
\cite[Proposition 5.6]{peror}.
 \end{rema}

Recall that a non-zero element $x$ in a monoid $M$ is termed
\emph{infinite} provided that there is a non-zero element $y$ such
that $x=x+y$. The (non-zero) element $x$ is \emph{properly infinite}
if $2x\leq x$ (whence $mx\leq x$ for all $m\in\mathbb{N}$). It follows
from the definitions that, if $x\leq y$ and $x$ is infinite, then so
is $y$. If $x$ is properly infinite and $x\le y$, then $y$ need not be
properly infinite. However, if we also assume that $y\propto x$, then
$y$ is necessarily properly infinite. (Indeed, if $x\le y\le nx$ and
$x$ is properly infinite, then $2y\le 2nx\le x\le y$.)

Evidently, every properly infinite element is infinite. It is well
known that the converse also holds in the simple case. Indeed, if $M$
is simple and  $x\in M$ is infinite, then $x=x+y$ for some non-zero
$y$ in $M$. Hence $x=x+my$ for all $m\in \mathbb{N}$. Now, $y$ is an
order unit, whence there is some $m\in\mathbb{N}$ such that $2x\leq my\leq
x$, and so $x$ is properly infinite.

Recall that a (non-zero) projection $p$ in a C$^*$-algebra is
\emph{properly infinite} if $p\oplus p\precsim p$, which is of course
equivalent to the element $[p]\in \V(A)$ being properly
infinite as defined above.

We record the following easy facts for future reference:
\begin{lem}
\label{lem:piou}
Let $M$ be a conical monoid with a properly infinite order unit $u$. Then
\begin{enumerate}
\item There is a properly infinite order unit $v$ with $u\leq v$ and $v+u=u$.
\item For any other order unit $w$, there is $n\in\mathbb{N}$ such that $nw$ is properly infinite.
\end{enumerate}
\end{lem}
\begin{proof}
(i). Since $2u\leq u$, there is an element $t$ such that $2u+t=u$. Now
put $v=u+t$.

(ii). Given an order unit $w$, there is $n \in \mathbb{N}$ with $u\leq
nw$. Since $u$ is properly infinite, $nw\leq u$. Now, $2nw\leq
u+u\leq u\leq nw$.
\end{proof}

Observe that, if $M$ is an abelian monoid, $I$ is an ideal in $M$,
and $u\in
I$, then $u$ is infinite (respectively, properly infinite) in $M$ if
and only if $u$ is infinite (respectively, properly infinite) in
$I$. Also, if $u\notin I$ and $u$ is properly infinite, then
$\overline{u}$ is properly infinite in $M/I$.

\begin{prop}
\label{prop:wdou}
Let $M$ be a conical refinement monoid with a  properly infinite
order unit. Then $M$ has weak divisibility for order units.
\end{prop}
\begin{proof}
Let $u$ be a properly infinite order unit in $M$, and let $v$ be any
other order unit. By Lemma \ref{lem:piou}, there is $n\in\mathbb{N}$
such that $nv$ is properly infinite.

If $v$ is not weakly divisible, then, by \cite[Theorem 6.7]{AGPS},
$\overline{v}$ is an atom in some quotient $M/I$ of $M$ (so in
particular $v\notin I$). Since $\langle \overline{v}\rangle$ is
isomorphic to $\mathbb{Z}^+$, we see that $\overline v$ and
$n\overline v$ are necessarily finite. But this contradicts the fact
that $nv$ is properly infinite in $M$ and hence also in $M/I$.
\end{proof}

\begin{corol}
Let $A$ be a properly infinite C$^*$-algebra with real rank zero. Then $\V(A)$ has weak divisibility of order units.
\end{corol}

\section{Refinement of order units}

The property of refinement in an abelian monoid is preserved under the
passage to ideals, as is well known and easy to show. However,
refinement will often
be lost when considering submonoids of the original monoid. One of our aims
in this section is to show that, under the presence of weak
divisibility, the subsemigroup $M^*$ of order units of a refinement
monoid $M$ will also have the refinement property. From this, it
follows that $M^*\cup \{0\}$ is a \emph{simple} refinement
monoid. This will allow us to reduce the proof of some of our results
to the simple case. The lemma below is due to Ken Goodearl.

\begin{lem}\label{lemma_1_1}
Let $M$ be a conical refinement monoid and suppose that we are
given a refinement
\[
\begin{array}{c|cc}
& c_1& c_2 \\
\hline a & a_1& a_2 \\
 b & b_1& b_2
\end{array}
\]
in $M$, such that $c_1\ge c_2$. Then there exists a refinement
\[
\begin{array}{c|cc}
& c_1& c_2 \\
\hline a & a'_1& a'_2 \\
 b & b'_1& b'_2 \end{array}
\]
such that $a_1' \ge a_2'$. If $b_2$ is an order unit, the refinement above can be taken so that furthermore $b'_2$ is an order unit.
\end{lem}

\begin{proof}
We have $c_1=c_2+x$ for some $x$, whence $a_1+b_1= a_2+ (b_2+x)$.
The latter equation has a refinement
$$\begin{array}{c|cc}
 & a_2& b_2+x \\
\hline a_1 & y_{11}& y_{12} \\
 b_1 & y_{21}& y_{22} \end{array}\,.$$ Now set $a_1'= y_{11}+y_{12}+y_{21}$ and
$a_2'= y_{11}$, while $b_1'= y_{22}$ and $b_2'= b_2+y_{21}$. Then
\[
\begin{array}{c|cc}
  & c_1& c_2 \\
\hline a & a'_1& a'_2 \\
 b & b'_1& b'_2\end{array}
\]
is a refinement with $a_1' \ge a_2'$. Moreover $b'_2\geq b_2$, so that $b_2'$ is an order unit if $b_2$ is.
\end{proof}

Recall that $M^*$ stands for the subsemigroup of order units of a conical abelian monoid, and that $x\leq^* y$  means that $x+z=y$ for some $z$ in $M^*$.

\begin{lem}
\label{lem:dummy}
Let $M$ be a conical refinement monoid with weak divisibility of order units. Suppose we have a refinement
\[
\begin{array}{c|cc}
  & c& d \\
\hline a & x_{11}& x_{12} \\
 b & x_{21}& x_{22}\end{array}
\]
in $M$, where $x_{11}, x_{22}\in M^*$ (hence also $a,b,c,d\in M^*$). Then, there exists a refinement
\[
\begin{array}{c|cc}
  & c& d \\
\hline a & z_{11}& z_{12} \\
 b & z_{21}& z_{22}\end{array}
\]
with $z_{ij}$ in $M^*$ for all $i,j$.
\end{lem}
\begin{proof}
By Corollary \ref{cor:oubelow}, there is an order unit $u$ such that $u\leq^* x_{11}, x_{22}$, so we may write $x_{11}=u+z_{11}$ and $x_{22}=u+z_{22}$, where both $z_{11}$ and $z_{22}$ belong to $M^*$. Now define $z_{12}=u+x_{12}$ and $z_{21}=u+x_{21}$. It is now a simple matter to check that $z_{ij}$ are order units that give the refinement of the conclusion.
\end{proof}

\begin{lem}
\label{lem:second}
Let $M$ be a conical refinement monoid with weak divisibility of order units. If we have
\[
a+b=c+d\,,
\]
where $a,b,c,d\in M^*$, there is then a refinement matrix
\[
\begin{array}{c|cc}
  & c& d \\
\hline a & z_{11}& z_{12} \\
 b & z_{21}& z_{22}\end{array}
\]
with $z_{11}\in M^*$.
\end{lem}
\begin{proof}
Use Lemma \ref{lemma_3_2} to split $a=x+y$ and $c=u+v$ where $u,v,x,y\in M^*$ and $v\leq u$, $y\leq x$. By Lemma \ref{lemma_1_1}, there is a refinement
\[
\begin{array}{c|cc}
  & v& u+d \\
\hline y & x_{11}& x_{12} \\
 b+x & x_{21}& x_{22}\end{array}
\]
with $x_{11}\leq x_{21}$, so in particular $x_{21}$ is an order unit. A second application of Lemma \ref{lemma_1_1} yields a refinement
\[
\begin{array}{c|cc}
  & v& u+d \\
\hline y & x'_{11}& x'_{12} \\
 b+x & x'_{21}& x'_{22}\end{array}
\]
where $x'_{21}\in M^*$ and $x'_{11}\leq x'_{12}$, so in particular $x'_{12}\in M^*$. We may thus apply Lemma \ref{lem:dummy} and assume at the outset that the elements $x_{ij}$ are all order units.

Refine the equality $b+x=x_{21}+x_{22}$ and get a refinement matrix
\[
\begin{array}{c|cc}
  & x_{12}& x_{22} \\
\hline x & a_{11}& a_{12} \\
 b & a_{21}& a_{22}\end{array}\,.
\]
Put $z'_{11}=x_{11}+a_{11}$, $z'_{12}=x_{12}+a_{12}$, so we obtain a refinement
\[
\begin{array}{c|cc}
  & v& u+d \\
\hline a & z'_{11}& z'_{12} \\
 b & a_{21}& a_{22}\end{array}\,,
\]
where $z'_{11}, z'_{12}\in M^*$. Proceeding in the same way for the
equality $u+d=z'_{12}+a_{22}$, we obtain the desired result.
\end{proof}

Of course, we can arrange, in Lemma \ref{lem:second}, any of the
entries of the final refinement to be an order unit. Arranging them
all to be order units \emph{simultaneously} is somewhat more delicate,
and we deal with this in the result below.

\begin{theor}
\label{thm:refou}
Let $M$ be a conical refinement monoid with weak divisibility of order units. If we have
\[
a+b=c+d\,,
\]
where $a,b,c,d\in M^*$, there is then a refinement matrix
\[
\begin{array}{c|cc}
  & c& d \\
\hline a & z_{11}& z_{12} \\
 b & z_{21}& z_{22}\end{array}
\]
with $z_{ij}\in M^*$ for all $i,j$. In particular, $M^*\cup\{0\}$ is a simple refinement monoid.
\end{theor}
\begin{proof}
Retain notation and procedure in the proof of the previous lemma up to the refinement
\[
\begin{array}{c|cc}
  & v& u+d \\
\hline a & z'_{11}& z'_{12} \\
 b & a_{21}& a_{22}\end{array}\,,
\]
where $z'_{11}$ and $z'_{12}$ are both order units.

Since $v\leq u$, we may apply Lemma \ref{lemma_1_1} to obtain a refinement matrix
\[
\begin{array}{c|cc}
  & v& u+d \\
\hline a & t_{11}& t_{12} \\
 b & t_{21}& t_{22}\end{array}\,,
\]
where $t_{11}\in M^*$ and $t_{21}\leq t_{22}$, so in particular also $t_{22}\in M^*$. Applying Lemma \ref{lem:dummy} if necessary, we may assume then that all $t_{ij}$'s are order units, by probably loosing the inequality $t_{21}\leq t_{22}$.

Next, refine $u+d=t_{12}+t_{22}$ applying Lemma \ref{lem:second} to obtain
\[
\begin{array}{c|cc}
  & u& d \\
\hline  t_{12}& w_{11}& w_{12} \\
 t_{22} & w_{21}& w_{22}\end{array}\,,
\]
with $w_{22}\in M^*$. Finally, put $z_{11}=t_{11}+w_{11}$, $z_{12}=w_{12}$, $z_{21}=t_{21}+w_{21}$ and $z_{22}=w_{22}$, where at least $z_{11}$ and $z_{22}$ are order units. A final application of Lemma \ref{lem:dummy} yields the desired result.
\end{proof}
\begin{rema}
\label{rem:simple}
In the simple atomless case, our result above says that a non-zero refinement problem $a+b=c+d$ (with $a,b,c,d$ being non-zero) admits a non-zero refinement matrix (i.e., with all entries being non-zero), a result which is well-known and much easier to prove.
\end{rema}

\begin{corol}
\label{cor:wehrung}
Let $M$ be a conical refinement monoid with weak divisibility of order units. If $a,b,c\in M^*$, $n\geq 1$, and $a+b=nc$, there are order units $u_1,\ldots, u_n$, $v_1,\ldots, v_n$ and $z_1,\ldots, z_{n-1}$ such that $u_i=u_{i+1}+z_i$ and
$v_i+z_i=v_{i+1}$ for every $i=1,\ldots,n-1$, and such that
\[
\begin{array}{ c|cccc}
  & u & u &\cdots & u \\
\hline  nu & u_1 & u_2 & \cdots & u_n \\
 v & v_1 & v_2 & \cdots & v_n \end{array}\,,
\]
is a refinement matrix.
\end{corol}
\begin{proof}
By Theorem \ref{thm:refou}, $M^*\cup\{0\}$ is a simple refinement monoid, whence Lemma 1.9 in \cite{WE} and its proof can be applied to obtain a refinement in $M^*$ (that is, in $M^*\cup\{0\}$ with all elements being non-zero) such as the the one in our statement.
\end{proof}

We now draw another consequence of Theorem \ref{thm:refou}.
\begin{defis}
\label{def:comparison}
{\rm We remind the reader that a monoid $M$ is \emph{almost unperforated} if whenever $(k+1)x\leq ky$ for $k\in\mathbb{N}$, it follows that $x\le y$. More generally, $M$ has $n$-comparison
if whenever $x,y_0,y_2, \dots,y_n$ are elements in $M$ such that $x
<_s y_j$ for all $j=0,1,\dots,n$, then $x \le y_0+y_1+ \dots
+y_n$. Here $x <_s y$ means that $(k+1)x \le ky$ for some natural
number $k$. It follows immediately from the definitions that $M$ is
almost unperforated if and only if $M$ has $0$-comparison. Notice that if $M$ has $n$-comparison, then $M$ has $m$-comparison for any $m\geq n$

Recall also that a monoid $M$ is termed \emph{strictly unperforated} if, whenever $nx+z=ny$ for $x,y,z\in M$ with $z\neq 0$ and $n\in \mathbb{N}$, there is a non-zero element $w$ such that $x+w=y$.}
\end{defis}

That these properties are equivalent for conical simple refinement monoids is quite possibly well-known. We state the result and outline the proof for the sake of completeness.

\begin{prop}
\label{lem:equiv}
Let $M$ be a simple conical refinement monoid with order unit $u$. Then, the following conditions are equivalent:
\begin{enumerate}
\item $M$ is almost unperforated.
\item $M$ has $n$-comparison for any $n$.
\item $M$ has $n$-comparison for some $n$
\item $M$ is strictly unperforated.
\item For every non-zero element $x$ in $M$ such that $x\leq u$, there exists $k\in\mathbb{N}$ such that, if $y\in M$ and $ky\leq u$, then also $y\leq x$.
\end{enumerate}
\end{prop}
\begin{proof}
If $M$ is atomic, then it is isomorphic to the infinite cyclic monoid,
by \cite[Lemma 1.6]{AP}, and all five conditions are easily seen to
hold in that case. We may thus assume that $M$ is non-atomic, and in
particular it will be atomless. This implies that $M$ is weakly
divisible (see Corollary~\ref{lem:simple}), so it has weak
divisibility for order units.

It is clear that (i) $\Rightarrow $ (ii) $\Rightarrow$ (iii), and also
that (iv) $\Rightarrow$ (i). That (iv) and (v) are equivalent follows
from \cite[Lemma 3.7]{perint}.

Let us check, finally, that (iii) $\Rightarrow$ (v). Assume $M$ has
$n$-comparison for some $n$. Let $x\in M$ be non-zero and assume
$x\leq u$. Since $M$ is weakly divisible, we may write $x=x_0+\cdots
+x_n$ with  all $x_i$ non-zero (hence all order units). There is
$k'\in\mathbb{N}$ such that $u\leq k'x_i$ for all $i$. Now let
$k=k'+1$. If $y\in M$ and $ky\leq u$, then $(k'+1)y\leq k'x_i$ for all
$i$, whence $n$-comparison implies that $y\leq x_0+\cdots+x_n=x$.
\end{proof}

Condition (v) above was termed \emph{weak comparability} (see \cite{AP}).

\begin{lem}{\rm (cf. \cite[Lemma 5.1 (a)]{PA})}\label{lemma_3_3}
Let $M$ be a conical refinement monoid with weak divisibility of
order units. Given order units $x_1,\ldots,x_k\in M$ and $n\in
\mathbb{N}$, there exists an order unit $y\in M$ such that $ny\leq^*
x_i$ for all $i$.
\end{lem}
\begin{proof}
This follows applying Theorem \ref{thm:refou}, which implies that $M^*\cup\{0\}$ is a simple refinement monoid, and then using condition (a) of Lemma 5.1 in \cite{PA}.
\end{proof}

\begin{corol}
\label{cor:cancellation}
Let $M$ be a conical refinement monoid with weak divisibility of order units. If $M^*\cup\{0\}$ is strictly unperforated, then $M^*$ is a cancellative monoid. This holds in particular if $M$ is strictly unperforated.
\end{corol}
\begin{proof}
By Theorem \ref{thm:refou}, we have that $M^*\cup\{0\}$ is a simple, conical, refinement monoid. We may then use Proposition \ref{lem:equiv}, together with \cite[Theorem 1.7]{AP}, to conclude that $M^*$ is cancellative.

Let us now check that if $M$ is strictly unperforated, then so is $M^*\cup\{0\}$. Suppose that $nx+z=ny$, for $x,y\in M^*\cup\{0\}$ and $z\in M^*$ (so clearly $y\neq 0$). If $x=0$, then $0+y=y$.

If $x\in M^*$ , then using Lemma \ref{lemma_3_3}, find an order unit $w$ such that $nw\leq^* z$. This implies then that $n(x+w)+z'=ny$, for some non-zero element $z'$, hence strict unperforation in $M$ implies $x+w+w'=y$, for some non-zero element $w'$. Since $w+w'$ is an order unit, we see that $M^*\cup\{0\}$ is strictly unperforated.
\end{proof}

\begin{corol} Let $A$ be a C$^*$-algebra with real rank zero. If $\V(A)$ has weak divisibility for order units and is strictly unperforated, then the subsemigroup of equivalence classes of full projections is cancellative.
\end{corol}

We close by developing divisibility results for not necessarily simple
conical refinement monoids with weak divisibility of order units. In
the simple atomless case (where weak divisibility is automatic by
Corollary~\ref{lem:simple}), these were obtained in \cite{PA},
although there only Riesz decomposition was assumed. We remark that
our results below hold true in this more general context, but we shall
not need this.

\begin{theor} {\rm{(cf.\ \cite[Theorem 5.2]{PA})}}
\label{thm:dvou}
Let $M$ be a conical refinement monoid with weak divisibility of order units. If $p$ and $r$ are order units, and $m\in\mathbb{N}$, there are order units $q$ and $s$ with $s\leq r$, $q$, and $p=mq+s$.
\end{theor}
\begin{proof}
By Theorem \ref{thm:refou}, $M^*\cup\{0\}$ is a simple refinement monoid, so \cite[Theorem 5.2]{PA} applies.
\end{proof}

The proof of the following result is identical to the one given in the simple case. Simplicity there is only assumed to ensure that the element $r$ in the statement below is an order unit.

\begin{lem} {\rm (cf.\ \cite[Lemma 5.1 (b)]{PA})}
\label{lem:3_3} Let $M$ be a conical refinement monoid, and let $p,r\in M$ with $r$ an order unit. Given $m\in \mathbb{N}$, there exist
$q,s\in M$ such that $p=mq+s$ with $s\leq (m-1)r$.
\end{lem}

\begin{prop}\label{theorem_3_5}
Let $M$ be a conical refinement monoid with weak divisibility of
order units. Let $p,r\in M$, with $r$ an order unit, and let $m\in
\mathbb{N}$. Then $p=mq+s$ for some $q$ and $s$ in $M$ such that $s\leq
r$.
\end{prop}
\begin{proof}
By Lemma \ref{lemma_3_3}, there exists an order unit $r'\in M$ such that $(m-1)r'<r$. Now
apply Lemma \ref{lem:3_3} to $p$ and $r'$. Then there exist
$q$ and $s$ in $M$ such that $p=mq+s$ with $s\leq (m-1)r'<r$, and the result follows.
\end{proof}

\section{Properly infinite order units}

In this section we prove a decomposition result for properly infinite order units in a refinement monoid that will be important for our applications to the corona factorization property for C$^*$-algebras discussed in the next section. We will greatly benefit from the results in the previous section, that allow us to reduce to the simple case.
We begin with a technical lemma:


\begin{lem}
\label{lem:induction} Let $M$ be a conical simple refinement
monoid. Suppose we are given non-zero elements $s,t,a_1,a_2, b_1, b_2,
z$ in $M$ and $n\in \mathbb{N}$ such that
\begin{enumerate}
\item $ns+a_1+a_2$ and $nt+b_1+b_2$ are infinite elements.
\item $a_1+z=s$ and $t+z=b_1$.
\end{enumerate}
Then, there are elements $s',t', p,q$ in $M$ such that
$$s'+t'=s+t, \qquad s'=a_1+p, \qquad  t'=t+q, \qquad z = p+q,$$
and such that the two elements  $(n+1)s'+a_2$ and $(n+1)t'+b_2$ are
infinite.
\end{lem}

\begin{proof}
Put $b=nt+b_1+b_2$, which by assumption is an infinite element (hence
properly infinite). We may thus write $b=b+c$ with $c$ an infinite
element. There is a refinement
\[
\begin{array}{ c|ccccc}
  & t & \cdots & t & b_1 & b_2\\
\hline  b & d_1 & \cdots & d_n & d_{n+1} & d_{n+2} \\
 c &e_1 & \cdots & e_n & e_{n+1} & e_{n+2} \end{array}\,,
\]
where all entries are non-zero, except possibly $d_{n+2}$ and $e_{n+2}$.

Use \cite[Theorem 5.2]{PA} to find elements $x$, $y$ in $M$ such that
$z=(n+1)x+y$, with $y\leq e_n$. Set $p=nx+y$, $q=x$, and
\[
s'=a_1+nx+y = a_1+p, \qquad  t'=x+t =t+q.
\]
Then $p+q=z$ and $s'+t'=a_1+t+z = s+t$. Further,
\begin{align*}
(n+1)s'+a_2 & =(n+1)a_1+(n+1)nx+(n+1)y+a_2\\ & \geq
(n+1)a_1+n(n+1)x+ny+a_2 \\ & =(n+1)a_1+nz+a_2=ns+a_1+a_2\,,
\end{align*}
so that $(n+1)s'+a_2$ is infinite. Also,
\begin{align*}
(n+1)t'+b_2 & =(n+1)t+(n+1)x+b_2\\ & \geq
d_1+\cdots+d_{n-1}+d_n+e_n+t+(n+1)x+b_2 \\ & \geq d_1+\cdots+d_{n-1}+d_n+y+t+(n+1)x+b_2\\
& =d_1+\cdots+d_n+t+z+b_2\\ & \geq d_1+\cdots+d_{n+2}=b\,,
\end{align*}
whence $(n+1)t'+b_2$ is infinite, as desired.
\end{proof}

\begin{prop}\label{lemma_2_4}
Let $M$ be a simple, conical  refinement monoid, and let $u$ be a non-zero element of $M$. Suppose that $nu$ is infinite for some $n\in\mathbb{N}$. Then there exist $s,t\in M$ with $u=s+t$ such that $ns$ and $nt$ are infinite.
\end{prop}
\begin{proof}
Write $nu=nu+v$ where $v\in M$ is infinite. Then using
\cite[Lemma 1.9]{WE} and its proof (see also Corollary \ref{cor:wehrung}) we can find non-zero elements $u_1\geq \cdots\geq u_n$ and
$v_1\leq \cdots\leq v_n$ with non-zero complements $u_i=u_{i+1}+z_i$ and
$v_i+z_i=v_{i+1}$ for every $i=1,\ldots,n-1$ such that
\[
\begin{array}{ c|cccc}
  & u & u &\cdots & u \\
\hline  nu & u_1 & u_2 & \cdots & u_n \\
 v & v_1 & v_2 & \cdots & v_n \end{array}\,
\]
is a refinement matrix.

We show by induction that for each $k=1,2,\dots,n$ there are elements
$s_k$ and $t_k$ in $M$ such that $s_k+t_k=u$, and the two elements
$$ks_k + (u_{k+1} + \cdots + u_n), \qquad kt_k + (v_{k+1} + \cdots +
v_n)$$
are infinite. Moreover, if $k < n$, then there is an element $z'_k$
such that $u_{k+1}+z'_k = s_{k}$ and $t_{k} + z'_{k} =
v_{k+1}$. For $k=n$ the two elements displayed above are equal to
$ns_n$ and $nt_n$, respectively, and so it will follow that $s=s_n$
and $t=t_n$ have the desired properties.

For $k=1$ we can take $s_1=u_1$, $t_1=v_1$, and $z'_1 = z_1$. Assume
that $1 \le k < n$ and that $s_k$, $t_k$, and $z'_k$ have been found. Then
it follows from Lemma~\ref{lem:induction} that there are elements
$s_{k+1}, t_{k+1}, p,q \in M$ such that
$$s_{k+1}+t_{k+1} = s_k+t_k=u,  \quad s_{k+1} = u_{k+1}+p, \quad
t_{k+1} = t_k + q, \quad p+q=z'_k,$$
and such that $(k+1)s_{k+1}+(u_{k+2} + \cdots + u_n)$ and
$(k+1)t_{k+1}+(v_{k+2} + \cdots + v_n)$ are infinite. If $k < n-1$, then
put $z'_{k+1} = z_{k+1}+p$ and calculate:
\begin{eqnarray*}
u_{k+2} + z'_{k+1}&=& u_{k+2} + z_{k+1} + p \; = \; u_{k+1} + p \; = \;
s_{k+1},\\
t_{k+1} + z'_{k+1} &=& t_k+q+z_{k+1}+p \, = \, t_k + z'_k + z_{k+1} \,
= \, v_{k+1} + z_{k+1} \, = \, v_{k+2}.
\end{eqnarray*}
\end{proof}




\begin{theor}
\label{prop:nonsimple} Let $M$ be a conical refinement monoid. Let $u$
be an order unit such that $nu$ is properly infinite for some
$n\in\mathbb{N}$. Then there exist order units $s,t\in M$ with
$u=s+t$, and such that $ns$ and $nt$ are properly infinite.
\end{theor}
\begin{proof}
We first note that $M$ has weak divisibility of order units, by Proposition \ref{prop:wdou}. Thus Theorem \ref{thm:refou} applies to conclude that $M^*\cup\{0\}$ is a simple refinement monoid. Since $nu$ is properly infinite (hence infinite as an element of $M^*\cup\{0\}$), we may use Proposition \ref{lemma_2_4} to find elements $s$ and $t$ in $M^*$ with $u=s+t$, and such that $ns$ and $nt$ are properly infinite in $M^*\cup\{0\}$, and so also in $M$.
\end{proof}

\begin{corol}\label{corollary_1_6}
Let $M$ be a conical refinement monoid. If $u$ is an order unit
and $nu$ is properly infinite for some $n\in \mathbb{N}$,  then there
is a sequence $t_1,t_2,t_3, \dots$ of order units in $M$ such that
$$t_1+t_2+ \cdots + t_k \le u$$ for all $k$, and $nt_i$ is properly
infinite for all $i$.
\end{corol}
\begin{proof}
By Theorem \ref{prop:nonsimple}, there are order units $s_1$ and $t_1$ such that $u=s_1+t_1$ and $ns_1$ and $nt_1$ are properly infinite. Continuing inductively, we may split each
$s_i=s_{i+1}+t_{i+1}$ for some order units $s_{i+1},t_{i+1}$ such that
$ns_{i+1}$ and $nt_{i+1}$ are properly infinite.
\end{proof}

\section{The Corona Factorization property in Monoids}

\noindent Recall that a \Cs{} $A$ is said to have the Corona
Factorization Property (CFP) if and only if every
full projection in $\cM(A\otimes \cK)$, the multiplier algebra of $A
\otimes \cK$, is properly infinite. We shall in this section translate
the CFP into a comparability property of the \Cs{} $A$ itself
(rather than its multiplier algebra)---under the assumption that $A$ is
of real rank zero. We begin by phrasing this comparability
property at the level of monoids.

\begin{defi}
A conical monoid $M$ is said to have the
\emph{strong Corona Factorization Property} (strong CFP)
if whenever $x, y_1,y_2,\ldots$
are elements in $M$ and $m$ is a natural number, one has
$$\forall n \in \N: x \le my_n \implies \exists k \in \N: x \le y_1+y_2+
\cdots + y_k.$$
\end{defi}

\noindent The property of monoids that matches the Corona
Factorization Property for \Cs s of real rank zero is weaker than the
property defined above (see Theorem~\ref{thm:CFP}). The strong CFP
considered above is perhaps more natural to
study than the weaker one defined below, and it also matches a
property of \Cs s (see
Theorem~\ref{thm:CFP2} below). To define the weaker notion of the
Corona Factorization
Property we first need the following:

\begin{defi}
A sequence $\{x_n\}_{n=1}^\infty$ in a monoid $M$ is said to
be \emph{full} if it is increasing and if for every $y \in M$ there
are natural numbers $n$ and $m$ such that $y\leq mx_n$.
\end{defi}

\begin{rema}
Every countable monoid $M=\{m_0,m_1,\ldots\}$ has a
full sequence. Indeed, the sequence $\{x_n\}$ given by
$x_n:=m_0+\cdots+m_n$ does the job.

The constant sequence $\{x_n\}$, where $x_n=x$ for all $n$, is full if
and only if the element $x$ is an order unit.

Full sequences are introduced to make up for the fact
that a monoid need not contain an order unit.
\end{rema}

\begin{defi}
A conical monoid $M$ is said to satisfy the
\emph{Corona Factorization Property} (CFP) if for every full
sequence $\{x_n\}$ in $M$, for every
sequence $\{y_n\}$ in $M$, and for every natural number $m$, one
has
$$\forall n \in \N: x_n \le my_n \implies  \exists
k \in \N:  x_1 \le y_1+y_2+ \cdots + y_k.$$
\end{defi}

\noindent It is clear that every (conical) monoid satisfying the
strong CFP also satisfies the CFP. The two notions clearly agree for
simple (conical) monoids.

If $M$ has the CFP and $\{x_n\}$, $\{y_n\}$, and $m$ are as above with
$x_n \le my_n$ for all $n$, then for all natural numbers $n$ and $k$
there exists a natural number $l$ such that $x_n \le y_k + y_{k+1} +
\cdots + y_l$. (To see this, one needs only consider the case where $k
\ge n$. Next, as $x_n \le x_k$, it suffices to consider the case where
$n=k$. Now apply the CFP to the sequences $\{x_i\}_{i \ge n}$ and
$\{y_i\}_{i \ge n}$.)

\begin{exem} \label{ex:n-comp}
Every \emph{almost unperforated} conical monoid satisfies the strong
CFP. More generally, if $M$ is a conical monoid which has
\emph{$n$-comparison} for some natural number $n$, then $M$ has the
strong CFP (see Definition \ref{def:comparison}).

Indeed, suppose that $M$ has $n$-comparison, suppose that $x,y_1,y_2, \dots$ are
elements in $M$, and $m$ is a natural number such that $x \le
my_j$ for all $j$. Put
$$z_j = y_{j(m+1)+1} + y_{j(m+1)+2} + \cdots + y_{j(m+1)+m+1}, \qquad
j=0,1,\dots,n.$$
Then $(m+1)x \le mz_j$, whence $x <_s z_j$ for all $j$, which by the
definition of $n$-comparison
implies that
$$x \le z_0+z_1 + \cdots +z_n = y_1+y_2 + \cdots + y_{n(m+1)+m+1}.$$
This shows that $M$ has the strong CFP.
\end{exem}

\noindent We now relate the CFP of a monoid to the CFP of a \Cs, and
hence we express the CFP for \Cs s in terms of a comparability
property of the \Cs. First we need two (well-known) lemmas about
comparison of projections in a multiplier algebra:

\begin{lem} \label{P-p}
Let $A$ be a $\sigma$-unital stable \Cs, let $P$ be a properly infinite, full
projection in $\cM(A)$, and let $p \le P$ be a projection in $A$. Then
$P-p$ is properly infinite and full in $\cM(A)$.
\end{lem}

\begin{proof} The assumptions on $P$ and $A$ imply that $P \sim
  1$, i.e., that $P = SS^*$ for some isometry $S$ in $\cM(A)$. Upon
  replacing $p$ by $S^*pS \in A$ we may assume that $P=1$.

Note that $(1-p)A(1-p)$ is $\sigma$-unital because $A$ is. Thus we may
apply \cite[Corollary~4.3]{HR} (and its proof) to conclude that
$(1-p)A(1-p)$ is
  stable. Hence $1-p$, being the unit of the multiplier algebra of the
  stable \Cs{} $(1-p)A(1-p)$, is properly infinite. Again using that
  $A$ is stable, by \cite[Theorem~3.3]{HR} (and its proof), we find
  that $p \precsim
  1-p$. Hence $1 \sim (1-p) \oplus p \precsim (1-p) \oplus (1-p)$
  which shows that $1-p$ is full in $\cM(A)$.
\end{proof}

\begin{lem} \label{comp-multiplier}
Let $A$ be a \Cs, let $\{p_n\}$ and $\{q_n\}$ be sequences of pairwise
orthogonal projections in $A$ such that the sums $P= \sum_{n=1}^\infty
p_n$ and $Q= \sum_{n=1}^\infty q_n$ are strictly convergent in the
multiplier algebra $\cM(A)$, and hence define projections $P$ and $Q$
in $\cM(A)$.
\begin{enumerate}
\item Suppose that there are sequences $\{k_n\}$ and
  $\{l_n\}$ of natural numbers such that $1 \le k_1 < l_1 < k_2 <
  l_2 < k_3 < \cdots$, and such that
$$[p_n] \le [q_{k_n}] + [q_{k_n+1}] + \cdots + [q_{l_n}]$$
for all $n$. Then $P \precsim Q$ in $\cM(A)$.
\item If $P \precsim Q$ in $\cM(A)$, then for every natural number $k$
  there exists a natural number $l$ such that
$$[p_1]+[p_2] + \cdots + [p_k] \le [q_1]+[q_2]+
  \cdots + [q_l]$$
in $\V(A)$.
\end{enumerate}
\end{lem}

\begin{proof}
(i). For each $n$, let $s_n \in A$ be a partial isometry with
$$s_n^*s_n = p_n, \qquad
s_ns_n^* \le q_{k_n} + q_{k_n+1} + \cdots + q_{l_n}.$$
As the sums $\sum p_n$ and $\sum q_n$ are strictly convergent, it
follows that the sum $S = \sum_{n=1}^\infty s_n$ is strictly
convergent in $\cM(A)$. Hence $P = S^*S \sim SS^* \le Q$.

(ii). Suppose that we are given a partial isometry $S \in \cM(A)$
such that $P = S^*S$ and $SS^* \le Q$. Put
$$s_0 = S(p_1+p_2 + \cdots + p_k).$$
Then $s_0$ is a partial isometry in $A$ satisfying $s_0^*s_0 = p_1+p_2
+ \cdots + p_k$ and $s_0s_0^* \le Q$. A standard argument, see e.g.\
\cite[Lemma~4.4]{Ro}, now shows that $p_1+p_2 + \cdots + p_k \precsim
q_1+q_2+ \cdots + q_l$ for some natural number $l$.
\end{proof}

\begin{theor} \label{thm:CFP}
Let $A$ be a $C^*$-algebra such that $A \otimes \cK$ has a countable
approximate unit consisting of projections.
\begin{enumerate}\item
If $A$ has the Corona Factorization Property (for \Cs s), then $\V(A)$ has the
Corona Factorization Property (for monoids).
\item Suppose that $A$ is of real rank
zero. Then $A$ has the Corona Factorization Property (for \Cs s)
if and only if
$\V(A)$ has the Corona Factorization Property (for monoids).
\end{enumerate}\end{theor}

\begin{proof}
We may identify $\V(A \otimes \cK)$ with $\V(A)$, and hence, upon
replacing $A$ with $A \otimes \cK$, we may assume that $A$ is
stable and that $A$ has a countable approximate unit consisting of
projections.

(i). Suppose that $\V(A)$ does not satisfy the CFP, i.e., there exist
a full sequence $\{x_n\}$ in $\V(A)$, another sequence $\{y_n\}$ in
  $\V(A)$, and a natural number $m$ such that $x_n\leq my_n$ for every
  $n$, while
$x_1 \nleqslant y_1+y_2+\cdots+y_k$ for every natural number $k$.

Take sequences of pairwise orthogonal projections
$\{p_n\}$ and $\{q_n\}$ in $A$ with $[q_n]=x_n$ and $[p_n]=y_n$ for
every $n$, and such that the sums $Q:=\sum^{\infty}_{n=1}q_n$ and
$P:=\sum^{\infty}_{n=1}p_n$ are strictly convergent in $\cM(A)$, and
hence define projections $Q$ and $P$ in $\cM(A)$.
We claim that $Q$ is equivalent to
$1=1_{\cM(A)}$. By the assumption that $A$ has a countable
approximate unit consisting of projections, we can write
$1 =\sum^{\infty}_{n=1} e_n$ (with the sum being strictly convergent)
for a suitable sequence $\{e_n\}$ of pairwise orthogonal projections
in $A$. Since $\{x_n\}$ is a full sequence, there exist
natural numbers $m_j$ and $k_j$ such that
$$[e_j] \le m_j x_{k_j} \le x_{k_j} + x_{k_j+1} + \cdots +
x_{k_j+m_j-1}.$$
Upon replacing each $k_j$ with a larger number we can assume that
$k_{j+1} \ge k_j+m_j$. It now follows from Lemma~\ref{comp-multiplier}~(i)
that $1
\precsim Q$ in $\cM(A)$. As $1$ is properly infinite and
$K_0(\cM(A))=0$ it follows that $1 \sim Q$ as claimed.

The $m$-fold direct sum $P \oplus P \oplus \cdots \oplus P$ is
equivalent to the projection $\sum_{n=1}^\infty p_n^\times$, where
$\{p_n^\times\}$ is a sequence of pairwise orthogonal
projections in $A$ whose sum converges strictly, and where
$[p_n^\times]=m[p_n]$ for all $n$. As $[q_n]=x_n \le m y_n =
[p_n^\times]$ we conclude from Lemma~\ref{comp-multiplier}~(i) that $1
\precsim P \oplus P \oplus \cdots \oplus P$, whence $P$ is full in
$\cM(A)$.

We finally observe that $Q \nprecsim P$ in $\cM(A)$. Otherwise, by
Lemma~\ref{comp-multiplier}~(ii), we would
have $x_1 = [q_1]\leq [p_1]+[p_2] + \cdots +[p_k] = y_1+y_2+ \cdots
+y_k$ for some natural number $k$, contradicting the hypothesis.

We have now shown that $P$ is a full projection in $\cM(A)$, and that
$P$ is not properly infinite (otherwise, $P$ would dominate any other
projection in $\cM(A)$).
Hence $A$ does not have  the Corona Factorization Property.

(ii). Suppose that $A$ is
of real rank zero, and that $\V(A)$ has
the CFP. We show that $A$ has the CFP. Take a full projection $P\in
\cM(A)$, and let $m$ be a natural number such that the $m$-fold direct
sum $P \oplus P \oplus \cdots
\oplus P$ is properly infinite (and hence equivalent to 1). We must show
that $P$ itself is properly infinite.

As in (i), write $1=1_{M(A)}=\sum_{n=1}^{\infty}e_n$.
Since $A$ is $\sigma$-unital and of real rank zero, the hereditary
sub-algebra
$PAP$ has a countable approximate unit consisting of
projections. (Indeed, if $\{e_n\}$ is an approximate unit for $A$,
then $\{Pe_nP\}$ is an approximate unit for $PAP$, whence $PAP$ is
$\sigma$-unital, and hence has a countable approximate unit consisting
of projections, because it also is of real rank zero.) It
follows that
we can write $P = \sum_{n=1}^{\infty} p_n$ (the sum being strictly
convergent), where the $p_n$'s are pairwise orthogonal projections in
$PAP$. As in the proof of (i), the $m$-fold direct sum $P \oplus P
\oplus \cdots \oplus P$ is
equivalent to a projection in $\cM(A)$ of the form
$\sum_{n=1}^\infty p_n^\times$, where
$[p_n^\times]=m[p_n]$ for all $n$.

Put $x_n = [e_1] + [e_2] + \cdots + [e_n]$. Then $\{x_n\}$ is a full
sequence in $\V(A)$. Indeed, let
$f$ be an arbitrary projection in $A$. Then $f \precsim 1_{\cM(A)}$,
whence $[f] \le x_n$ for some $n$ by
Lemma~\ref{comp-multiplier}~(ii).

We proceed to show that there is a sequence $1 = k_1 < k_2 < k_3 <
\cdots$ of natural numbers such that
\begin{equation} \label{eq:1}
x_n \le [p_{k_n}^\times] + [p_{k_n+1}^\times] + \cdots +
[p_{k_{n+1}-1}^\times] = m\big([p_{k_n}] + [p_{k_n+1}] + \cdots +
[p_{k_{n+1}-1}]\big)
\end{equation}
for all $n$. The existence of $k_2$ such that \eqref{eq:1} holds for
$n=1$ follows from Lemma~\ref{comp-multiplier}~(ii) applied to the
relation $1 = \sum_{n=1}^{\infty}e_n
\precsim \sum_{n=1}^{\infty} p_n^\times$. To establish \eqref{eq:1}
for $n=2$ use Lemma~\ref{P-p} to see that $\sum_{n=k_2}^\infty
p_n^\times$ is properly infinite and full. Applying
Lemma~\ref{comp-multiplier}~(ii) to the resulting
relation $1 = \sum_{n=1}^{\infty}e_n \precsim  \sum_{n=k_2}^\infty
p_n^\times$ yields $k_3 > k_2$ such that \eqref{eq:1} holds for
$n=2$. Continue in this manner to find the remaining $k_j$'s.

Put $y_n = [p_{k_n}] + [p_{k_n+1}] + \cdots + [p_{k_{n+1}-1}]$. Then
$x_n \le m y_n$ for all $n$ by \eqref{eq:1}. We next claim that there
exists a sequence $1 = l_1 < l_2 < l_3 < \cdots$ of natural numbers
such that
\begin{equation} \label{eq:2}
x_n \le y_{l_n} + y_{l_n+1} + \cdots + y_{l_{n+1}-1}
\end{equation}
for all $n$. The existence of $l_2$ such that \eqref{eq:2} holds for
$n=1$ follows directly from the assumption that $\V(A)$ has the
CFP. Now apply the CFP to the sequences $\{x_n\}$ and $\{y_n'\}$ where
$y_n' = y_{n+l_1}$, noting that $x_n \le x_{n+l_1} \le my_{n+l_1} =
my_n'$. Then we get $l_3 > l_2$ such that \eqref{eq:2} holds for
$n=2$. Continue in this manner to find the remaining $l_j$'s.

Put $j_n = k_{l_n}$. Then, by \eqref{eq:2}, we have
$$[e_n] \le x_n \le [p_{j_n}] + [p_{j_n}+1] + \cdots +
[p_{j_{n+1}-1}]$$
for all $n$. We can now conclude from Lemma~\ref{comp-multiplier}~(i)
that $1 \precsim P$, whence $P$ is properly infinite as desired.
\end{proof}

\noindent Combine the result above with Example~\ref{ex:n-comp} to get the
following:

\begin{corol} Let $A$ be a $\sigma$-unital \Cs{} of real rank zero, and
  suppose that $\V(A)$ is almost unperforated, or that $\V(A)$ has
  $n$-comparison for some natural number $n$. Then $A$ has the Corona
  Factorization Property.
\end{corol}




\noindent We proceed to study permanence properties of the strong CFP
and of the CFP, in particular with respect to passing to ideals.
Recall the definition of ideals (see below
Corollary~\ref{cor:oubelow}). If $S$ is a subset of $M$, then
the set $I(S)$ of all elements $x \in M$, such that $x \le k(y_1+y_2+
\dots + y_r)$ for some elements $y_1, y_2, \dots, y_r$ in $S$ and some
natural number $k$, is an ideal in $M$. We refer to $I(S)$ as the
\emph{ideal generated} by $S$. An ideal generated by a countable subset $S$
of $M$ is said to be countably generated. If $I=I(S)$ is generated by
the countable set $S = \{y_1,y_2,y_3, \dots\}$, then it is also
generated by thet set $\{y_1,y_1+y_2,y_1+y_2+y_3, \dots\}$, and so
every countably generated ideal is generated by an increasing sequence
in $M$.
Clearly, all ideals in a countable monoid are countably
generated.

As usual, $I \triangleleft M$ is short for saying that $I$ is an ideal
in $M$.

Every ideal in a conical refinement monoid is itself a conical
refinement monoid.

\begin{prop} Let $M$ be a conical monoid and let $I$ be an ideal in
  $M$.
\begin{enumerate}
\item If $M$ has the strong CFP, then so does the quotient monoid $M/I$.
\item If $M$ has the CFP and if $I$ is countably generated, then $M/I$
  has the CFP.
\end{enumerate}
\end{prop}

\begin{proof}
Let $x \mapsto \overline{x}$ denote the quotient mapping $M \to
M/I$. For $x,y \in M$ one has $\overline{x} \le \overline{y}$ in $M/I$
if and only if there exists $z \in I$ such that $x \le y+z$ in $M$.

(i). Suppose that $M$ has the strong CFP. Let $x,y_1,y_2,
\dots$ be elements in $M$ and let $m$ be a natural number such that
$\overline{x} \le m \overline{y}_n$ for all $n$. Then $x \le m
y_n+z_n$ for some $z_n \in I$. Hence $x \le m(y_n + z_n)$ for all $n$,
whence
$x \le (y_1+z_1) + (y_2+z_2) + \cdots + (y_k+z_k)$
for some $k$ (because $M$ has the strong CFP). This shows that
$\overline{x} \le \overline{y}_1 + \overline{y}_2 + \cdots +
\overline{y}_k$, and we conclude that $M/I$ has the strong CFP.

(ii). Suppose now that $M$ has the CFP. Let us first remark that each full
sequence in $M/I$ lifts to a full sequence in $M$. Indeed, let
$\{x_n\}$ be a sequence in $M$ such that $\{\overline{x}_n\}$ is full
in $M/I$. Then $x_n \le x_{n+1}+z_{n+1}$ for some $z_{n+1} \in
I$. Accordingly, if we set $x'_n = x_n + (z_1+z_2+ \cdots + z_n)$,
then $\overline{x'_n} = \overline{x}_n$ and $x'_1 \le x'_2 \le x'_3
\le \cdots$ in $M$. Let $\{v_n\}$
be an increasing sequence which generates the ideal $I$, and put
$x''_n = x'_n + v_n$. Then $\overline{x''_n} = \overline{x}_n$ and
$\{x''_n\}$ is a full sequence in $M$. To see the latter note first
that  $\overline{x''_n}$ is increasing by construction. Let $y \in M$
be given. Then $\overline{y} \le m_1 \overline{x}_{n_1}$ for some natural
numbers $n_1$ and $m_1$. It follows that $y \le m_1 x'_{n_1} + w$ for some $w$
in $I$. Next, $w \le m_2 v_{n_2}$ for some natural numbers $m_2$ and
$n_2$. Put $m = \max\{m_1,m_2\}$ and $n = \max\{n_1,n_2\}$. Then
$$y \le m_1x'_{n_1} + w \le  m_1x'_{n_1} + m_2 v_{n_2} \le mx'_n +
mv_n = mx''_n.$$
This shows that $\{x''_n\}$ is full in $M$.

Suppose that we are given a full sequence in $M/I$. By the
argument above we may write this full sequence as
$\{\overline{x}_n\}$, where $\{x_n\}$ is a full sequence in $M$. Let
$\overline{y}_1, \overline{y}_2, \overline{y}_3, \dots$ be another
sequence in $M/I$ and let $m$ be a natural number such that
$\overline{x}_n \le m \overline{y}_n$ for all $n$. Then $x_n \le my_n
+ z_n \le m(y_n+z_n)$ for some $z_n$ in $I$. As $M$ has the CFP we
conclude that $x_1 \le (y_1+z_1) + (y_2+z_2) + \cdots + (y_k+z_k)$ for
some natural number $k$. It follows that $\overline{x}_1 \le
\overline{y}_1+\overline{y}_2 + \cdots + \overline{y}_k$. This shows
that $M/I$ has the CFP.
\end{proof}

\noindent One can deduce from this proposition that the
quotient of any separable \Cs{} of real rank zero with the CFP again
has the CFP. This, however, is well-known, cf.\ \cite{KN}.

The CFP does not pass to ideals. Indeed, if $M$ is any
conical monoid, then we can consider the monoid $M^\sharp:=M\sqcup
\{\infty\}$, where addition is given by $u + \infty = \infty$ for all
$u$ in $M^\sharp$. Then $M^\sharp$ has the CFP. Indeed, let
$\{x_n\}$ and $\{y_n\}$ be sequences in $M^\sharp$ such that
$\{x_n\}$ is full and $x_n \le m y_n$ for some $m$. Then
$x_k=y_k=\infty$ for all sufficiently large $k$, whence $x_1 \le
y_1+y_2 + \cdots + y_k$ if $k$ is chosen large enough so that $y_k=
\infty$. However, $M$ is an ideal in $M^\sharp$, and being an
arbitrary conical monoid, $M$ need not have the CFP (see e. g.\
Proposition~\ref{exem_1} below).

The situation is different for the strong CFP. We omit the trivial
proof of the proposition below.

\begin{prop} \label{prop:CFP-ideals1}
If $M$ is a conical monoid with the strong Corona
  Factorization Property, then every ideal in $M$ also has the strong
  Corona Factorization Property.
\end{prop}

\noindent In the case of conical refimenent monoids
the strong CFP and the CFP are related as follows:

\begin{prop} \label{prop:CFP-ideals2}
Let $M$ be a conical refinement monoid. Then $M$ has the strong Corona
Factorization Property if and only if every ideal in $M$
has the Corona Factorization Property.
\end{prop}

\begin{proof} If $M$ has the strong CFP, then so does every ideal $I$ in
  $M$, cf.\ Proposition~\ref{prop:CFP-ideals1}, and if $I$ has the
  strong CFP, then $I$ also has the CFP. (Note that in this direction
  of the proof we have not used that $M$ is a refinement monoid.)

To prove the reverse direction we need to establish the following
fact. Suppose that $x, y \in M$ and $m \in \N$ are such that
$x \le my$. Then
there is $y'$ in $M$ such that $x \le my'$, $y' \le x$, and $y' \le
y$. Indeed, $my = x+t$ for some $t$ in $M$. Hence, by
\cite[Lemma 1.9]{WE}, we have a
refinement
$$\begin{array}{c|cccc}
  & y & y &\cdots & y \\
\hline x & a_1 & a_2 & \cdots & a_m \\
  t & b_1 & b_2 & \cdots & b_m
\end{array}$$
with $a_1\leq a_2\leq\cdots\leq a_m$ (and
$b_1\geq b_2\geq\cdots\geq b_m$). Put $y'=a_m$. Then $y' \le x$, $y'
\le y$ and $my' \ge a_1+a_2 + \cdots + a_m = x$.

Assume now that every ideal in $M$ has the CFP (actually we need only
assume that every singly generated ideal in $M$ has the CFP). We wish
to show that $M$ has the strong CFP. Let $x,y_1,y_2, \dots$ in $M$ and
$m \in \N$ be given such that $x \le my_n$ for every $n$. Let $I$ be
the ideal in $M$ generated by $x$. Then $x$ is an order unit in $I$, so the
constant sequence $\{x_n\}$, with $x_n=x$ for all $n$, is a full
sequence in $I$. Use the fact established above to find $y'_n$ in $M$
such that $y'_n \le x$, $y'_n \le y_n$, and $x \le m y'_n$ for all
$n$. Then each $y'_n$ belongs to $I$, and $x_n = x \le my'_n$ for all
$n$. Since $I$ is assumed to have the CFP there is a natural number
$k$ such that
$$x = x_1 \le y'_1 + y'_2 + \cdots + y'_k \le y_1 + y_2 + \cdots +
y_k.$$
This proves that $M$ has the strong CFP.
\end{proof}

\noindent We can now characterize the class of \Cs s $A$ (of real rank zero)
for which the monoid $\V(A)$ has the strong CFP:

\begin{theor} \label{thm:CFP2}
 Let $A$ be a $\sigma$-unital \Cs{} of real rank zero. Then
  $\V(A)$ has the strong Corona Factorization Property (for monoids) if
  and only if every ideal $I$ in $A$ has
  the Corona Factorization Property (for \Cs s).
\end{theor}

\begin{proof} The assumptions on $A$ imply that $\V(A)$ is a conical
 refinement monoid. The mapping $I \mapsto \V(I) \subseteq \V(A)$ from
 ideals $I$ in $A$ into subsets of $\V(A)$ gives a bijection
 between ideals in $A$ and ideals in $\V(I)$ (in the case where $A$ is
 of real rank zero). Each ideal in $A$ is separable and of real
 rank zero. Hence, by Theorem~\ref{thm:CFP}, the condition that every
 ideal $I$ in $A$ has the CFP (for \Cs s) is equivalent to the condition that
 every ideal in $\V(A)$ has the CFP (for monoids). The theorem therefore
 follows from Proposition~\ref{prop:CFP-ideals2} above.
\end{proof}

\noindent We end this section by discussing an important consequence
of having the CFP:

\begin{theor} \label{finite-infinite} Let $M$ be a conical refinement
  monoid, and let $u$ be an element in $M$ such that $mu$ is properly
  infinite for some natural number $m$. Then $u$ itself is properly
  infinite if one of the two conditions below hold:
\begin{enumerate}
\item $M$ has the strong CFP,
\item $M$ has the CFP and $u$ is an order unit in $M$.
\end{enumerate}
\end{theor}

\begin{proof} Upon replacing $M$ by the ideal generated by $u$ (and
  using Proposition~\ref{prop:CFP-ideals2}) it suffices to consider
  case (ii). By Corollary~\ref{corollary_1_6} there
  is a sequence $t_1,t_2,t_3, \dots$ of order units in $M$ such
  that $t_1+t_2+ \cdots + t_k \le u$ for all $k$, and $mt_n$ is
  properly infinite for all $n$. It follows in particular that $u \le
  mt_n$ for all $n$. Now, because $M$ has the CFP, there exists $k$
  such that $u \le t_1+t_2+ \cdots + t_k$. Applying the CFP again we
  obtain that $u \le t_{k+1} + t_{k+2} + \cdots + t_l$ for some $l >
  k$. This shows that
$$2u \le t_1+t_2 + \cdots + t_k + t_{k+1} + t_{k+2} + \cdots + t_l \le
u,$$
whence $u$ is properly infinite.
\end{proof}

\begin{corol} \label{cor:CFP}
Let $A$ be a $\sigma$-unital \Cs{} of real rank zero, and let
  $p$ be a projection in $A$ such that the $m$-fold direct sum $p
  \oplus p \oplus \cdots \oplus p$ is properly infinite (in $M_m(A)$)
  for some natural number $m$. Then $p$ itself is properly infinite if
  one of the following two conditions below hold:
\begin{enumerate}
\item Every ideal in $A$ has the CFP,
\item $A$ has the CFP and $p$ is a full projection in $A$.
\end{enumerate}
\end{corol}

\begin{proof} This follows from Theorem~\ref{finite-infinite} above
  together with Theorems~\ref{thm:CFP} and \ref{thm:CFP2}. Use also
  that if $p$ is a projection in $A$, then $p$ is properly infinite
  (as a projection in $A$) if
  and only if $[p]$ is properly infinite (as an element in $\V(A)$).
\end{proof}

\noindent It follows from this result that \Cs s with the CFP satisfy
the following dichotomy: either all full projections are properly
infinite or no full projections are properly infinite.

The corollary below was proved in an unpublished paper by
S.\ Zhang. (One can drop the separability assumption by passing to a
 suitable separable sub-\Cs.)

\begin{corol} A separable simple \Cs{} of real rank zero with the Corona
  Factorization Property is either stably finite or purely infinite.
\end{corol}

\begin{proof} Suppose that $A$ is a separable simple \Cs{} of real
  rank zero and with the CFP. Then the monoid $\V(A)$ is a simple,
  countable refinement monoid with the CFP (by
  Theorem~\ref{thm:CFP}). If all elements in $\V(A)$ are finite, then
  all projections in $A$ and in matrix algebras over $A$ are finite,
  whence $A$ is stably finite. Suppose that $\V(A)$ contains an infinite
  element. Then $\V(A)$ contains a properly infinite element $u$
  (because all infinite elements in a simple monoid are properly infinite). Let
  $v$ be any other non-zero element in $\V(A)$. Then $v$ is an order unit (because $\V(A)$ is simple), and so $u \le mv$ for some natural
  number $m$. It follows that $mv$ is properly infinite (because $mv +
  mv \le u \le mv$, the former inequality holds because $u$ is
  a properly infinite order unit). By Theorem~\ref{finite-infinite} we get
  that $v$ is itself properly infinite. This shows that
   all non-zero elements of $\V(A)$
  are properly infinite. This translates into the statement that all
  non-zero projections in $A$ (and in matrix algebras over $A$) are
  properly infinite. As $A$ is assumed to be
of real rank zero, this implies that
  $A$ is purely infinite.
\end{proof}

\noindent It is an open problem if all \Cs s of real rank zero have
the CFP. At the level of monoids, the
analog problem has a negative answer.

\begin{prop}\label{exem_1}
There exists a simple conical refinement monoid $M$ with an order unit $u$ such that $u$ is finite while $2u$ is properly infinite. In
particular, $M$ does not have the CFP.
\end{prop}

\begin{proof}
The monoid $M_0= \{0,u,\infty\}$, where $u+u= \infty$, is simple and
conical. The element $u$ is an order unit in $M_0$, $u$ is finite,
while $2u = \infty$ is infinite.
By \cite{Wehrung_Embedding}, $M_0$ is a unitary submonoid
of a simple conical refinement monoid $M$. Obviously $2u$ remains
infinite in $M$. Since $M_0$ is cofinite in $M$, we see that $u$ is
an order unit in $M$. If $u+x=u$ for some $x\in M$, then $x\in M_0$
because $M_0 \subseteq M$ is unitary, whence $x=0$. Therefore $u$
remains finite in $M$.

The last conclusion follows from Theorem~\ref{finite-infinite}.
\end{proof}

\noindent The monoid $M_0$ from the proof of Proposition~\ref{exem_1}
above has some intesting features. It is a simple monoid with a finite
element $u$ such that $2u$ is properly infinite. At the same time it
does satisfy the strong CFP. (This is easy to check, there are not so
many ways in which one can choose the elements $x,y_1,y_2, \dots$.)
This shows that in Theorem~\ref{finite-infinite} one cannot omit the
assumption that the monoid has the refinement property.

\section*{Acknowledgements}

The first and second named authors were partially supported by a
MEC-DGESIC grant (Spain) through Project MTM2008-0621-C02-01/MTM, and
by the Comissionat per Universitats i Recerca de la Generalitat de
Catalunya. The third named author was supported by a grant from the
Danish Natural Science Research Council (FNU). The second named author
wishes to thank G. del Peral for inspiring conversations concerning
the results in Section 3. Part of this research was carried out during
visits of the first and third named authors to UAB (Barcelona), of
the second named author to SDU (Odense), and of the two first
mentioned authors to Copenhagen. We wish to thank all parties
involved for the hospitality extended to us.

\end{document}